\documentclass{amsart}

\newif\ifdraft
\draftfalse
\usepackage{amsmath,amsfonts,amsthm,mathrsfs}
\usepackage{amssymb}

\usepackage[unicode,bookmarks]{hyperref}
\usepackage[usenames,dvipsnames]{xcolor}
\hypersetup{colorlinks=true,citecolor=NavyBlue,linkcolor=BrickRed,urlcolor=Orange}

\usepackage[alphabetic,initials]{amsrefs}

\usepackage{enumitem}

\usepackage{chngcntr}

\ifdraft
\usepackage[notcite,notref,color]{showkeys}

\definecolor{labelkey}{gray}{0.5}
\fi

\usepackage{tikz}

\usepackage{tikz-cd}
\tikzset{commutative diagrams/arrow style=math font}

\usetikzlibrary{matrix,arrows}
\newlength{\myarrowsize} 

\pgfarrowsdeclare{cmto}{cmto}{
	\pgfsetdash{}{0pt} 
	\pgfsetbeveljoin 
	\pgfsetroundcap 
	\setlength{\myarrowsize}{0.6pt}
	\addtolength{\myarrowsize}{.5\pgflinewidth}
	\pgfarrowsleftextend{-4\myarrowsize-.5\pgflinewidth} 
	\pgfarrowsrightextend{.8\pgflinewidth}
}{
	\setlength{\myarrowsize}{0.6pt} 
  	\addtolength{\myarrowsize}{.5\pgflinewidth}  
	\pgfsetlinewidth{0.5\pgflinewidth}
	\pgfsetroundjoin
	\pgfpathmoveto{\pgfpoint{1.5\pgflinewidth}{0}}
	\pgfpatharc{-109}{-170}{4\myarrowsize}
	\pgfpatharc{10}{189}{0.58\pgflinewidth and 0.2\pgflinewidth}
	\pgfpatharc{-170}{-115}{4\myarrowsize+\pgflinewidth}
	\pgfpathclose
	\pgfusepathqfillstroke
	\pgfpathmoveto{\pgfpoint{1.5\pgflinewidth}{0}}
	\pgfpatharc{109}{170}{4\myarrowsize}
	\pgfpatharc{-10}{-189}{0.58\pgflinewidth and 0.2\pgflinewidth}
	\pgfpatharc{170}{115}{4\myarrowsize+\pgflinewidth}
	\pgfpathclose
	\pgfusepathqfillstroke
	\pgfsetlinewidth{2\pgflinewidth}
}

\pgfarrowsdeclare{cmonto}{cmonto}{
	\pgfsetdash{}{0pt} 
	\pgfsetbeveljoin 
	\pgfsetroundcap 
	\setlength{\myarrowsize}{0.6pt}
	\addtolength{\myarrowsize}{.5\pgflinewidth}
	\pgfarrowsleftextend{-4\myarrowsize-.5\pgflinewidth} 
	\pgfarrowsrightextend{.8\pgflinewidth}
}{
	\setlength{\myarrowsize}{0.6pt} 
  	\addtolength{\myarrowsize}{.5\pgflinewidth}  
	\pgfsetlinewidth{0.5\pgflinewidth}
	\pgfsetroundjoin
	\pgfpathmoveto{\pgfpoint{1.5\pgflinewidth}{0}}
	\pgfpatharc{-109}{-170}{4\myarrowsize}
	\pgfpatharc{10}{189}{0.58\pgflinewidth and 0.2\pgflinewidth}
	\pgfpatharc{-170}{-115}{4\myarrowsize+\pgflinewidth}
	\pgfpathclose
	\pgfusepathqfillstroke
	\pgfpathmoveto{\pgfpoint{1.5\pgflinewidth}{0}}
	\pgfpatharc{109}{170}{4\myarrowsize}
	\pgfpatharc{-10}{-189}{0.58\pgflinewidth and 0.2\pgflinewidth}
	\pgfpatharc{170}{115}{4\myarrowsize+\pgflinewidth}
	\pgfpathclose
	\pgfusepathqfillstroke
	\pgfpathmoveto{\pgfpoint{1.5\pgflinewidth-0.3em}{0}}
	\pgfpatharc{-109}{-170}{4\myarrowsize}
	\pgfpatharc{10}{189}{0.58\pgflinewidth and 0.2\pgflinewidth}
	\pgfpatharc{-170}{-115}{4\myarrowsize+\pgflinewidth}
	\pgfpathclose
	\pgfusepathqfillstroke
	\pgfpathmoveto{\pgfpoint{1.5\pgflinewidth-0.3em}{0}}
	\pgfpatharc{109}{170}{4\myarrowsize}
	\pgfpatharc{-10}{-189}{0.58\pgflinewidth and 0.2\pgflinewidth}
	\pgfpatharc{170}{115}{4\myarrowsize+\pgflinewidth}
	\pgfpathclose
	\pgfusepathqfillstroke
	\pgfsetlinewidth{2\pgflinewidth}
}

\pgfarrowsdeclare{cmhook}{cmhook}{
	\pgfsetdash{}{0pt} 
	\pgfsetbeveljoin 
	\pgfsetroundcap 
	\setlength{\myarrowsize}{0.6pt}
	\addtolength{\myarrowsize}{.5\pgflinewidth}
	\pgfarrowsleftextend{-4\myarrowsize-.5\pgflinewidth} 
	\pgfarrowsrightextend{.8\pgflinewidth}
}{
	\setlength{\myarrowsize}{0.6pt} 
  	\addtolength{\myarrowsize}{.5\pgflinewidth}  
 	\pgfsetdash{}{0pt}
	\pgfsetroundcap
	\pgfpathmoveto{\pgfqpoint{0pt}{-4.667\pgflinewidth}}
	\pgfpathcurveto
    {\pgfqpoint{4\pgflinewidth}{-4.667\pgflinewidth}}
    {\pgfqpoint{4\pgflinewidth}{0pt}}
    {\pgfpointorigin}
	\pgfusepathqstroke
}

\newenvironment{diagram*}[2]{%
\[%
\begin{tikzpicture}[>=cmto,baseline=(current bounding box.center),%
	to/.style={->,font=\scriptsize,cap=round},%
	into/.style={cmhook->,font=\scriptsize,cap=round},%
	onto/.style={-cmonto,font=\scriptsize,cap=round},%
	math/.style={matrix of math nodes, row sep=#2, column sep=#1,%
		text height=1.5ex, text depth=0.25ex}]%
}{%
\end{tikzpicture}%
\]%
\ignorespacesafterend%
}

\newcommand{\Dmod}{\mathscr{D}}
\newcommand{\Mmod}{\mathcal{M}}

\newcommand{\derR}{\mathbf{R}}

\newcommand{\shH}{\mathcal{H}}

\newcommand{\tensor}{\otimes}

\newcommand{\ZZ}{\mathbb{Z}}
\newcommand{\QQ}{\mathbb{Q}}

\newcommand{\CC}{\mathbb{C}}

\DeclareMathOperator{\gr}{gr}
\DeclareMathOperator{\DR}{DR}

\newcommand{\shf}[1]{\mathscr{#1}}
\newcommand{\OX}{\shf{O}_X}

\def\overbar#1#2#3{{%
	\setbox0=\hbox{\displaystyle{#1}}%
	\dimen0=\wd0
	\advance\dimen0 by -#2 
	\vbox {\nointerlineskip \moveright #3 \vbox{\hrule height 0.3pt width \dimen0}%
		\nointerlineskip \vskip 1.5pt \box0}%
}}

\newcommand{\shF}{\shf{F}}
\newcommand{\shG}{\shf{G}}
\newcommand{\shE}{\shf{E}}

\newcommand{\shO}{\shf{O}}

\makeatletter
\let\@@seccntformat\@seccntformat
\renewcommand*{\@seccntformat}[1]{%
  \expandafter\ifx\csname @seccntformat@#1\endcsname\relax
    \expandafter\@@seccntformat
  \else
    \expandafter
      \csname @seccntformat@#1\expandafter\endcsname
  \fi
    {#1}%
}
\newcommand*{\@seccntformat@subsection}[1]{%
  \textbf{\csname the#1\endcsname.}
}
\makeatother

\makeatletter
\let\@paragraph\paragraph
\renewcommand*{\paragraph}[1]{%
	\vspace{0.3\baselineskip}%
	\@paragraph{\textit{#1}}%
}
\makeatother

\counterwithin{equation}{subsection}
\counterwithout{subsection}{section}
\counterwithin{figure}{subsection}

\newtheorem{theorem}[equation]{Theorem}
\newtheorem*{theorem*}{Theorem}
\newtheorem{lemma}[equation]{Lemma}
\newtheorem*{lemma*}{Lemma}
\newtheorem{corollary}[equation]{Corollary}
\newtheorem{proposition}[equation]{Proposition}
\newtheorem*{proposition*}{Proposition}

\theoremstyle{definition}
\newtheorem{definition}[equation]{Definition}
\newtheorem*{definition*}{Definition}
\theoremstyle{remark}
\newtheorem{remark}[equation]{Remark}

\newtheorem*{example*}{Example}
\newtheorem*{problem*}{Problem}

\theoremstyle{plain}

\newcommand{\theoremref}[1]{\hyperref[#1]{Theorem~\ref*{#1}}}
\newcommand{\lemmaref}[1]{\hyperref[#1]{Lemma~\ref*{#1}}}
\newcommand{\definitionref}[1]{\hyperref[#1]{Definition~\ref*{#1}}}
\newcommand{\propositionref}[1]{\hyperref[#1]{Proposition~\ref*{#1}}}
\newcommand{\conjectureref}[1]{\hyperref[#1]{Conjecture~\ref*{#1}}}
\newcommand{\corollaryref}[1]{\hyperref[#1]{Corollary~\ref*{#1}}}
\newcommand{\exampleref}[1]{\hyperref[#1]{Example~\ref*{#1}}}

\makeatletter
\let\old@caption\caption
\renewcommand*{\caption}[1]{%
	\setcounter{figure}{\value{equation}}%
	\stepcounter{equation}%
	\old@caption{#1}\relax%
}
\makeatother

\newcounter{intro}

\newtheorem{intro-conjecture}[intro]{Conjecture}
\newtheorem{intro-corollary}[intro]{Corollary}
\newtheorem{intro-theorem}[intro]{Theorem}

\newcommand{\parref}[1]{\hyperref[#1]{\S\ref*{#1}}}

\makeatletter
\newcommand*\if@single[3]{%
  \setbox0\hbox{${\mathaccent"0362{#1}}^H$}%
  \setbox2\hbox{${\mathaccent"0362{\kern0pt#1}}^H$}%
  \ifdim\ht0=\ht2 #3\else #2\fi
  }
\newcommand*\rel@kern[1]{\kern#1\dimexpr\macc@kerna}
\newcommand*\widebar[1]{\@ifnextchar^{{\wide@bar{#1}{0}}}{\wide@bar{#1}{1}}}
\newcommand*\wide@bar[2]{\if@single{#1}{\wide@bar@{#1}{#2}{1}}{\wide@bar@{#1}{#2}{2}}}
\newcommand*\wide@bar@[3]{%
  \begingroup
  \def\mathaccent##1##2{%
    \if#32 \let\macc@nucleus\first@char \fi
    \setbox\z@\hbox{$\macc@style{\macc@nucleus}_{}$}%
    \setbox\tw@\hbox{$\macc@style{\macc@nucleus}{}_{}$}%
    \dimen@\wd\tw@
    \advance\dimen@-\wd\z@
    \divide\dimen@ 3
    \@tempdima\wd\tw@
    \advance\@tempdima-\scriptspace
    \divide\@tempdima 10
    \advance\dimen@-\@tempdima
    \ifdim\dimen@>\z@ \dimen@0pt\fi
    \rel@kern{0.6}\kern-\dimen@
    \if#31
      \overline{\rel@kern{-0.6}\kern\dimen@\macc@nucleus\rel@kern{0.4}\kern\dimen@}%
      \advance\dimen@0.4\dimexpr\macc@kerna
      \let\final@kern#2%
      \ifdim\dimen@<\z@ \let\final@kern1\fi
      \if\final@kern1 \kern-\dimen@\fi
    \else
      \overline{\rel@kern{-0.6}\kern\dimen@#1}%
    \fi
  }%
  \macc@depth\@ne
  \let\math@bgroup\@empty \let\math@egroup\macc@set@skewchar
  \mathsurround\z@ \frozen@everymath{\mathgroup\macc@group\relax}%
  \macc@set@skewchar\relax
  \let\mathaccentV\macc@nested@a
  \if#31
    \macc@nested@a\relax111{#1}%
  \else
    \def\gobble@till@marker##1\endmarker{}%
    \futurelet\first@char\gobble@till@marker#1\endmarker
    \ifcat\noexpand\first@char A\else
      \def\first@char{}%
    \fi
    \macc@nested@a\relax111{\first@char}%
  \fi
  \endgroup
}
\makeatother

\newcommand{\Q}{\mathbb{Q}}

\newcommand{\V}{\mathbb{V}}

\newcommand{\Z}{\mathbb{Z}}

\newcommand{\cH}{\mathcal{H}}

\newcommand{\cL}{\mathcal{L}}

\newcommand{\cV}{\mathcal{V}}
\newcommand{\cW}{\mathcal{W}}

\newcommand{\bV}{\mathbf{V}}

\newcommand{\fddot}{F_\bullet}
\newcommand{\rarr}{\longrightarrow}

\begin{document}

\vspace{\baselineskip}

\title{Weak positivity for Hodge modules}

\author{Mihnea Popa}
\address{Department of Mathematics, Northwestern University,
2033 Sheridan Road, Evanston, IL 60208, USA} 
\email{\tt mpopa@math.northwestern.edu}

\author{Lei Wu}
\address{Department of Mathematics,  Northwestern University,
2033 Sheridan Road, Evanston, IL 60208, USA}
\email{\tt leiwu2014@u.northwestern.edu}


\thanks{MP was partially supported by the NSF grant DMS-1405516}


\setlength{\parskip}{0.19\baselineskip}

\maketitle

\begin{abstract}
We prove the weak positivity of the kernels of Kodaira-Spencer-type maps for pure Hodge module
extensions of generically defined variations of Hodge structure.
\end{abstract}


\subsection{Introduction}

Let $X$ be a smooth projective complex variety, and $U \subseteq X$ a dense open subset. It is of fundamental importance that (extensions of) variations of Hodge structure on $U$ come with inherent positivity properties. This study was initiated by Griffiths, who proved when  $U= X$ that the lowest term in the Hodge filtration has a semi-positive definite metric, so in particular is a nef vector bundle. This fact was extended by Fujita \cite{Fujita} (when $X$ is a curve) and Kawamata \cite{Kawamata} to include the case when $D = X - U$ is a simple normal crossings divisor, and the variation has unipotent monodromy along its components. Generalizations of this result were provided in recent work by Fujino-Fujisawa \cite{FF} and Fujino-Fujisawa-Saito \cite{FFS}. It is possible, and very useful for applications, to see these results as part of a wider picture involving all kernels of Kodaira-Spencer type maps associated to meromorphic connections with log-poles and unipotent monodromy. This was first considered by Zuo \cite{Zuo}, while recently a detailed study has been provided by Brunebarbe \cite{Brunebarbe1}; see Section \ref{log-WP}.

In this paper we give a further extension to the setting of Hodge modules. This is very convenient when dealing with arbitrary 
families of varieties; see for instance \cite{PS}, where the present results are used towards proving Viehweg's hyperbolicity conjecture for families of maximal variation, and more generally to put constrains on the spaces on which certain geometrically relevant Hodge modules can exist.

Let $\bV$ be a polarizable variation of Hodge structure on $U$, with quasi-unipotent local monodromies. By a fundamental theorem of M. Saito  \cite[\S3.b]{Saito-MHM}, $\bV$ admits a unique pure Hodge module extension $M$ with strict support $X$; conversely, any pure Hodge module with strict support $X$ is generically a variation of Hodge structure. See Section \ref{background} for further background.

We consider in particular the filtered left $\Dmod_X$-module $(\Mmod, F_{\bullet})$ underlying $M$. For each $p$, we have a natural Kodaira-Spencer type  $\shO_X$-module homomorphism
$$\theta_p:  {\rm gr}_p^F \Mmod \longrightarrow {\rm gr}_{p+1}^F \Mmod \otimes \Omega_X^1$$
induced by the $\Dmod$-module structure, and we denote 
$$K_p (M) : = {\rm ker}~ \theta_p.$$

\begin{intro-theorem}\label{thm:WP}
If $M$ is a polarizable pure Hodge module with strict support $X$, then the torsion-free sheaf $K_p (M)^\vee$ is weakly positive 
for any $p$.
\end{intro-theorem}

In the case when $D$ is a simple normal crossings divisor, and $\bV$ has unipotent monodromy along its 
components,  we will see in Section \ref{normal_crossings} that there is a close relationship between $K_p (M)$ and $K_p (\cV^{\ge 0})$, where $\cV^{\ge 0}$ is the Deligne canonical extension of $\bV$ and 
$K_p (\cV^{\ge 0})$ is defined  analogously. For this, the 
results of \cite{Zuo} and \cite{Brunebarbe}  on logarithmic connections can be applied to deduce weak positivity, as explained
in Section \ref{log-WP}. We will then proceed in Sections \ref{normal_crossings} and \ref{general} by successive reductions 
to this case, using some of the main results from Saito's theory.

A seemingly different weak positivity result was proved by Schnell and the first author using Kodaira-Saito vanishing, as explained in \cite[Theorem 1.4]{Schnell1} and \cite[Theorem 10.4]{Popa}: it states that the lowest non-zero graded piece $F_{{\rm low}} \Mmod$ in the filtration on a Hodge module extending a generic variation of Hodge  structure is weakly positive. (This includes Viehweg's result on $f_* \omega_{X/Y}$ for a morphism $f:X \rightarrow Y$ of smooth projective varieties.)  Using duality, one can in fact deduce this as a special case of \theoremref{thm:WP}; indeed, it is observed in \cite{Suh}, see also the end of \cite[Section 4]{Wu}, that $F_{{\rm low}} \Mmod$ can be related to the dual of ${\rm gr}_F^{{\rm top}} \Mmod(*D)$, the 
top non-zero graded piece of the localization of $\Mmod$ along $D$, which in turn coincides with $K_{{\rm top}} (M)$.\footnote{This is the analogue of deducing the Fujita-Kawamata semipositivity results as special cases of the result of Brunebarbe and Zuo cited above.} It would be interesting to relate \theoremref{thm:WP} to vanishing theorems as well.

\subsection{Background material}\label{background}
In this section we review basic terminology and facts regarding weak positivity, filtered $\Dmod$-modules, and Hodge modules.

\noindent
{\bf Weak positivity.}
We start by recalling the notion of weak positivity introduced by Viehweg \cite{Viehweg1}; it is a higher rank analogue of the notion of a pseudo-effective line bundle, known to have numerous important applications to birational geometry.

\begin{definition}
Let $X$ be a smooth quasi-projective variety. A torsion-free coherent sheaf $\shF$ on $X$ is \emph{weakly positive over
an open set $U \subseteq X$} if for every integer $\alpha > 0$ and every ample line bundle $H$ on $X$, there exists 
an integer $\beta > 0$ such that 
$$\widehat{S}^{\alpha \beta} \shF \otimes H^{\otimes \beta}$$
is generated by global sections at each point of $U$. It is simply called \emph{weakly positive} if such an open set $U$ exists.
Here the notation $\widehat{S}^k \shF$ stands for the reflexive hull of the sheaf $S^k \shF$.
\end{definition}

The following basic lemma will be used for detecting weak positivity.

\begin{lemma}[{\cite[Lemma 1.4]{Viehweg1}}]\label{wplemma}
Let $\shF$ and $\shG$ be torsion-free coherent sheaves on $X$. Then the following hold:

\noindent
(1) If $\shF \rightarrow \shG$ is surjective over $U$, and if $\shF$ is weakly positive over $U$, then $\shG$ is weakly positive over $U$. 

\noindent
(2) If $f: X\rightarrow Y$ is a birational morphism such that $f|_U$ is an isomorphism, and $E$ is a divisor supported on the exceptional locus of $f$ such that $\shF\otimes\shO_X(E)$ weakly positive over $U$, then $f_*\shF$ is weakly positive over $f(U)$.

\noindent
(3) If $\pi: X' \rightarrow X$ is a finite morphism and $\pi^*\shF$ is weakly positive over $\pi^{-1}(U)$, 
then $\shF$ is weakly positive over $U$. 
\end{lemma}

\noindent
{\bf Filtered $\Dmod$-modules and the de Rham complex.}
Let $X$ be a complex manifold, or a smooth complex algebraic variety, of dimension $n$. If $(\Mmod, \fddot)$ is a filtered left 
$\Dmod_X$-module, then the filtered de Rham complex of $(\Mmod, \fddot)$ is 
\[\text{DR}(\Mmod):=[\Mmod \rarr \Mmod \otimes \Omega^1_X  \rarr \cdots \rarr \Mmod \otimes \Omega^n_X] [n],\]
with filtration given by
\[F_p\text{DR}(\Mmod):=[F_p \Mmod \rarr F_{p+1} \Mmod \otimes \Omega^1_X \rarr \cdots \rarr F_{p+n}\Mmod \otimes \Omega^n_X][n].\]
The associated graded complexes for this filtration are 
\[\gr^F_p\text{DR}(\Mmod):=[\gr^F_p \Mmod \rarr \gr^F_{p+1} \Mmod \otimes \Omega^1_X \rarr \cdots \rarr \gr^F_{p+n}\Mmod \otimes \Omega^n_X][n].\]
These are complexes of coherent $\shO_X$-modules, placed in degrees $-n, \ldots, 0$.

\begin{definition}
The \emph{Kodaira-Spencer kernels} of the filtered $\Dmod_X$-module $(\Mmod, \fddot)$ are the coherent sheaves 
$$K_p (\Mmod): = {\rm ker} \big( \theta_p:  {\rm gr}_p^F \Mmod \longrightarrow {\rm gr}_{p+1}^F \Mmod \otimes \Omega_X^1\big)$$
where $\theta_p$ are the $\shO_X$-module homomorphisms considered above. Equivalently, 
\begin{equation}\label{left_coh}
K_p (\Mmod) \simeq \cH^{-n} \gr^F_p {\rm DR}(\Mmod).
\end{equation}
\end{definition}

\noindent
{\bf Hodge modules and variations of Hodge structure.}
Let $X$ be a smooth complex algebraic variety of dimension $n$, and let 
$$\bV = (\mathcal{V}, F_\bullet, \V_{\QQ})$$
be a polarizable variation of $\QQ$-Hodge structure of weight $k$ on an open set $U \subset X$.
Here $\V_{\QQ}$ is a local system of $\QQ$-vector spaces on $U$, 
$\mathcal{V} = \V_{\QQ}\otimes_{\QQ} \shO_U$, and $F_p = F_p \mathcal{V}$ an increasing\footnote{This convention is adopted in order to match the standard Hodge module terminology; usually one would consider a decreasing filtration where $F^p$ corresponds to our $F_{-p}$.} filtration of 
subbundles of $\mathcal{V}$ satisfying Griffiths transversality with respect to the connection associated to $\mathcal{V}$.

In \cite{Saito-MHP}, Saito associates to this data a pure Hodge module of weight $n+ k$ on $X$, whose 
main constituents are:

\noindent
(1) ~ A filtered regular holonomic left $\Dmod_X$-module $(\Mmod, F_{\bullet})$,  with $F_{\bullet} \Mmod$
a good filtration by $\OX$-coherent subsheaves,  whose restriction to $U$ is $\cV$ together with its connection and Hodge filtration.

\noindent
(2) ~ A $\QQ$-perverse sheaf $P$ on $X$ such that $\DR_X(\Mmod) \simeq P \tensor_{\QQ} \CC$.

Moreover, one of Saito's fundamental results \cite[\S3.b]{Saito-MHM}, states that there is a unique such extension if we impose the condition that $M$ have strict support $X$, i.e. not have any sub or quotient objects with support strictly smaller than $X$. Its underlying perverse sheaf is $P = {\rm IC}_Z( \V_{\QQ}) =  {}^pj_{!*}\V_\Q$, the intersection complex of the given local system, and therefore one sometimes uses the notation 
\[M : = j_{!*}\bV.\]
These are the main objects we consider in this paper; we refer to them as \emph{the} Hodge module extension of the generically defined VHS. In Section \ref{normal_crossings} we will give a more concrete description of the Hodge filtration in the case when the complement of $U$ is a simple normal crossings divisor.

Assuming that the complement of $U$ in $X$ is a divisor $D$, we can also consider the filtered $\Dmod_X$-module 
$(\cV (*D), F_{\bullet})$, where $\cV(*D)$ is Deligne's meromorphic connection extending $\cV$. This underlies a natural mixed Hodge module extension of $\bV$ introduced in \cite{Saito-MHM}, denoted $j^*j^{-1} M$, and sometimes called the localization along $D$. More precisely
\[ j_*j^{-1}M =(\cV(*D), \fddot, j_*\V_\Q),\]
We will also have a more concrete description of this Hodge module in the case when $D$ is a simple normal crossings divisor.

\subsection{Weak positivity for logarithmic variations of Hodge structure}\label{log-WP}
In this section we recall background on meromorphic connections with log poles, and the 
 results of \cite{Zuo} and  \cite{Brunebarbe1}  that are used in the proof of our main theorem.

\noindent
{\bf Logarithmic connections.} 
We begin by reviewing the theory of logarithmic connections; see for instance \cite[\S2]{EV} and \cite[\S5.2.2]{HTT}.  
Let $X$ be a smooth complex variety of dimension $n$, and let $D=\sum D_i$ be a reduced simple normal crossings divisor. 
We will call such an $(X, D)$ a smooth log pair. Write $U= X \smallsetminus D$, and denote the inclusion by
\[j: U\hookrightarrow X.\]
Suppose $\cV$ is a holomorphic vector bundle of finite rank on $X$. An integrable logarithmic connection on $\cV$ along $D$ is a $\CC$-linear morphism
\[\nabla:\cV \rarr  \cV \otimes \Omega^1_X (\text{log}~D)\]
satisfying the Leibniz rule and $\nabla^2=0$. By analogy with the definition above, the logarithmic de Rham complex is defined as
\[{\rm DR}_D (\cV):=[\cV \otimes \Omega^\bullet_X(\text{log} D)][n].\]
For each $D_i$, composing the Poincar\'e residue and $\nabla$ induces the residue map 
$$\Gamma_i \in \text{End}(\cV|_{D_i}).$$

Now given a $\CC$-local system $\V$ on $U$, $\cV=\V\otimes\shO_U$ is a vector bundle with integrable connection, and this can be extended to an integrable logarithmic connection on $X$. Such an extension is unique if the eigenvalues of $\Gamma_i$ are required to be in the image of a fixed section $\tau$ of the projection $\CC\rarr\CC / \Z$; see \cite[Theorem 5.2.17]{HTT}.  If $\tau$ is chosen so that the real parts belong to the interval $[0, 1)$, the corresponding extension of $\cV$ is the Deligne canonical extension, denoted by $\cV^{\geq0}$.

If $E=\sum\alpha_i D_i$ is any Cartier divisor supported on $D$, $\cV (E)$ is also an integrable logarithmic connection, with residue $\Gamma^E_i$ given locally by 
\begin{equation}\label{res}
\Gamma^E_i=\Gamma_i-\alpha_i\cdot\text{Id}.
\end{equation}
See for instance \cite[Lemma 2.7]{EV}.

\noindent
{\bf Logarithmic variations of Hodge structure.}
Following \cite{Brunebarbe}, it will be convenient to consider the following notion which combines 
logarithmic connections and variations of Hodge structure:

\begin{definition}[Log VHS]
Let $(X, D)$ be a smooth log pair with $X$ projective, and set $U:=X \smallsetminus D$. A log variation of Hodge structure (log VHS) 
along $D$ consists of the following data: 

\noindent
(1) A logarithmic connection $(\cV, \nabla)$ along $D$.

\noindent
(2) An exhaustive increasing filtration $\fddot$ on $\cV$ by holomorphic subbundles (the Hodge filtration)\footnote{Again, to be consistent with conventions for Hodge modules, we use increasing filtrations.}, satisfying the Griffiths transversality condition
\[\nabla F_p\subseteq  F_{p+1} \otimes \Omega_X^1(\text{log}~D).\]  

\noindent
(3) A $\Q$-local system $\V^U_\Q$ on $U$, such that $(\cV|_{U}, \fddot|_U, \V^U_\Q)$ is a variation of Hodge structure on 
$U$.

\noindent 
A log VHS is polarizable if the variation of Hodge structure defined on $U$ is so.\footnote{Polarizability as defined here is slightly different from the notion considered in \cite{Brunebarbe}; however they are equivalent when the residues of $\nabla$ are nilpotent.}
\end{definition}

Filtered logarithmic de Rham complexes for log variations of Hodge structure are defined just as in the case of filtered 
$\Dmod$-modules. We can also consider the associated graded quotients of the filtered de Rham complex,
\[\gr^F_p {\rm DR}_D(\cV) = [\gr^F_{p+\bullet} \cV \otimes \Omega_X^\bullet(\text{log}~D)][n],\]
which are $\shO_X$-linear complexes of holomorphic vector bundles in degrees $-n, \ldots, 0$.
In this context it is known that the duals of the Kodaira-Spencer-type kernels
$$K_p (\cV) : = \cH^{-n} \text{gr}^F_p\text{DR}_D(\cV)  \simeq 
{\rm ker} \big( {\rm gr}_p^F \cV \overset{\theta_p}{\longrightarrow} {\rm gr}_{p+1}^F \cV \otimes \Omega_X^1 ({\rm log}~D)\big)$$ 
satisfy a positivity property; this extends the well-known Fujita-Kawamata semi-positivity theorem, and is due to 
Zuo {\cite[Theorem 1.2]{Zuo}; see also Brunebarbe \cite[Theorem 0.6]{Brunebarbe1} (or \cite[Theorem 3.6]{Brunebarbe}}) for 
an alternate proof. 

\begin{theorem} \label{sp}
Let $(X, D)$ be a smooth log pair with $X$ projective, and let $\cV$ be the bundle with logarithmic connection underlying a polarizable log VHS with nilpotent residues along $D$. If $\shF$ is a holomorphic subbundle of $K_p (\cV)$, then 
$\shF^{\vee}$ is nef. 
\end{theorem}

For the proof of the next statement, it will be useful to note the following:  if
\[g: (Y, E)\rarr (X, D)\]
is a morphism of smooth projective log pairs, where  $E=(f^{*}D)_{\text{red}}$, then a simple local calculation shows that 

\noindent
(1) If $\cV$ underlies a (polarizable) log VHS along $D$, then so does $f^*\cV$, along $E$;

\noindent
(2) If furthermore the residues of $\cV$ along $D$ are nilpotent, then so are those of $f^*\cV$ along $E$.

As a consequence of \theoremref{sp} we have the following statement, which in the geometric case is essentially \cite[Lemma 4.4(v)]{VZ}.

\begin{corollary}\label{WP_log}
If $\cV$ underlies a polarizable log VHS with nilpotent residues along $D$, then $K_p (\cV)^\vee$ is weakly positive.
\end{corollary}
\begin{proof}
Let us denote for simplicity
$$\shE_p : = {\rm gr}_p^F \cV.$$
If $K_p (\cV)$ is already a subbundle of $\shE_p$, then we are done by \theoremref{sp}, as nef vector bundles 
are weakly positive. This need not be the case in general; however, a standard resolution of singularities argument 
applies to provide a birational morphism $\rho: X' \rightarrow X$ with $X'$ smooth projective, such that 
$\rho^* K_p (\cV)$ has a morphism to a subbundle  $\shF$ of $\rho^*\shE_p$, which is generically an isomorphism.

On the other hand, we have a commutative diagram 
$$
\begin{tikzcd}
\rho^*\shE_p \drar{\theta'_p} \rar{\rho^*\theta_p} &\rho^*\shE_{p+1} \otimes \rho^* \Omega_X^1(\text{log}~D) \dar{\psi} \\
   & \rho^*\shE_{p+1} \otimes \Omega_{X'}^1(\text{log}~E)
\end{tikzcd}
$$
where $E=(\rho^*D)_{\text{red}}$ and $\psi$ is induced by the natural map
\[\rho^* \Omega_X^1(\text{log}~D) \rarr \Omega_{X'}^1(\text{log}~E).\] 
Using the remark before the statement of the Corollary, $\theta^\prime_p$ again corresponds to a log VHS (induced by the same generic variation of Hodge structure), with a logarithmic connection along $E$. Since $\rho$ is birational, $\theta'_p(\shF)$ is generically 0, and so identically $0$ since it embeds in a vector bundle.  Hence $\shF$ is contained in $\text{Ker}(\theta'_p)$, and so $\shF^\vee$ is nef by \theoremref{sp}. By \lemmaref{wplemma}(1) we obtain that 
$\rho^* K_p (\cV)^\vee$ is weakly positive as well. Finally, this implies that $ K_p (\cV)^\vee$ itself is weakly positive, 
using \lemmaref{wplemma}(2).
\end{proof}

\subsection{Normal crossings case}\label{normal_crossings}
In this section we establish the main result for pure Hodge module extensions of variations of Hodge structure defined on the complement of a simple normal crossings divisor, by reducing to the result for logarithmic connections in the previous section.

\noindent
{\bf Hodge modules associated to variations of Hodge structure.}
Let $(X, D)$ be a smooth  log pair,  with $X$ projective and $\dim X = n$.
Denote $U=X \smallsetminus  D$ and $j: U\hookrightarrow X$. We consider a polarizable variation of Hodge structure 
\[\bV=(\cV, \fddot, \V_\Q)\]
over $U$, with quasi-unipotent local monodromies along the components $D_i$ of $D$. In particular the eigenvalues of all residues are rational numbers. For $\alpha \in \ZZ$,  we 
denote by $\cV^{\geq\alpha}$ (resp. $\cV^{>\alpha}$) the Deligne extension with eigenvalues of residues along the  $D_i$ in $[\alpha, \alpha+1)$ (resp. $(\alpha, \alpha+1]$). Recall that $\cV^{\geq\alpha}$ is filtered by
\begin{equation}\label{filtr}
F_p\cV^{\geq\alpha}=\cV^{\geq\alpha}\cap j_*F_p\cV,
\end{equation}
while the filtration on $\cV^{>\alpha}$ is defined similarly. The terms in the filtration are 
locally free by Schmid's nilpotent orbit theorem \cite{Schmid} (see also e.g. \cite[2.5(iii)]{Kollar} for the quasi-unipotent case); 
we have that $(\cV^{\geq\alpha(>\alpha)}, \fddot, \V_\Q)$ is a polarizable log VHS.

Following Section \ref{background}, let now $M$ be the pure Hodge module with strict support $X$ uniquely extending 
$\bV$. It is proved in \cite[\S 3.b]{Saito-MHM}  that
\[M=(\Dmod_X\cV^{>-1}, \fddot, j_{!*}\V_\Q),\] where
\[F_p\Dmod_X\cV^{>-1}=\sum_i F_i\Dmod_X \cdot F_{p-i}\cV^{>-1}.\]

We also consider the natural mixed Hodge module extension of $\bV$, namely the localization
\[j_*j^{-1}M =(\cV(*D), \fddot, j_*\V_\Q).\]
See e.g. \cite[\S4.2]{Saito-MHM}. We recall that $\cV(*D)$ is Deligne's meromorphic connection extending $\cV$, with lattice 
$\cV^{>\alpha}$ (or $\cV^{\geq\alpha}$) for any $\alpha\in\QQ$, namely
\[\cV(*D)=\cV^{>\alpha}\otimes \shO_X (*D),\]
with filtration given by
\[F_p\cV(*D)=\sum_i F_i\Dmod_X \cdot F_{p-i}\cV^{\geq-1}.\]

In \cite[Proposition 3.11(ii)]{Saito-MHM}, Saito constructed a filtered quasi-isomorphism that will be used in what follows, namely
\begin{equation}\label{qiso1}
\big([\cV^{\geq0}\otimes \Omega^\bullet_X(\text{log} ~D)][n], \fddot\big) \simeq \big(\text{DR}( j_*j^{-1}M ), \fddot \big).
\end{equation}
Here the notation on the right hand side refers to the filtered de Rham complex of the underlying filtered 
$\Dmod$-module.


\noindent
{\bf Unipotent reduction and proof in the normal crossings case.}
In the setting of the previous section, a standard argument using Kawamata's covering construction \cite{Kawamata} provides a finite flat morphism of smooth projective log pairs
\[f: (Y, E)\rarr (X, D)\]
with $(f^*D)_\text{red}=E$, such that the pull-back $\V_1: = f_1^* \V_\Q$ has unipotent local monodromies along all irreducible components $E_i$ of $E$, where $f_1=f|_{f^{-1}(U)}$. \par

To perform the reduction to the case of $f^* \bV$ on $(Y, E)$, let us first recall that one calls a 
lattice for $\cV (*D)$ any locally free sheaf $\cL$ on $X$ satisfying 
$$\cV(*D) \simeq \cL \otimes \shO_X(*D),$$
and preserved by the action of $f\cdot \nabla$, where $\nabla$ is the meromorphic connection on $\cV(*D)$
 and $f$ is a local equation of  $D$. We also consider the same notions for $E$.

For the argument, first, it is not hard to see that  $f^* \cV(*D)$ is a regular meromorphic connection on $Y$ extending $\V_1$. This  implies 
\[ f^* \cV(*D) \simeq \cV_1(*E),\]
e.g. by Deligne's Riemann-Hilbert correspondence for meromorphic connections with regular singularities 
(see \cite[Theorem 5.2.20]{HTT}). Moreover, $f^*\cV^{\geq0}$ is a lattice of $\cV_1(*E)$. 
A simple local calculation shows that the eigenvalues of the residue of the connection along each component of $E$ are nonnegative integers. On the other hand,
\[ \cV_1(*E)=\lim_{\longrightarrow}\cV_1^{\geq 0}(-kE)\]
over $k \in \ZZ$. Hence 
\[f^*\cV^{\geq 0}\subseteq \cV_1^{\geq 0}(- kE)=\cV_1^{\geq k},\] for some integer $k$.
We claim that $k \ge 0$, so that in particular 
\begin{equation}\label{placement}
f^*\cV^{\geq 0}\subseteq \cV_1^{\geq 0}.
\end{equation}
This is a special case of the following general statement:

\begin{proposition} 
Let $Y$ be a smooth complex variety and $E = \sum E_i$ a simple normal crossings divisor on $Y$. Suppose that $\cW$ is a lattice for a regular meromorphic connection with poles along $E$, and with quasi-unipotent local monodromies along all 
$E_i$. Denoting by $\Gamma_i$ the residue of $\cW$ along $E_i$ with respect to the connection, if \[t : = \textup{min}~\{ \textup{eigenvalues of } \Gamma_i \textup{ for all } i \},\] then 
\[\cW\subseteq \cW^{\geq t},\] where $\cW^{\geq t}$ is the Deligne extension with eigenvalues of all residues in $[t, t +1)$.
\end{proposition}
\begin{proof}
Suppose $k$ is the smallest integer such that $\cW\subseteq \cW^{\geq -k}$. By the Artin-Rees Lemma (see for instance \cite[Corollary A.36]{Kas}), there is an integer $\ell \geq 0$ such that 
\[\cW\subseteq \cW^{\geq -k}(-\ell E_i) \,\,\,\, {\rm and} \,\,\,\, \cW\not\subseteq \cW^{\geq -k} \big(-(\ell +1)E_i\big). \] 
Hence 
$$\cW/\cW\cap\cW^{\geq -k} \big(-(\ell +1)E_i \big) \neq 0,$$ 
and given that $\cW (- E_i) \subseteq \cW^{\geq -k} \big(-(\ell +1)E_i\big)$ we have a short exact sequence
\[0\rightarrow \dfrac{\cW\cap\cW^{\geq -k} \big(-(\ell+1)E_i \big)}{\cW(-E_i)}\rightarrow \cW|_{E_i}\rightarrow \dfrac{\cW}{\cW\cap\cW^{\geq -k}\big(-(\ell+1)E_i \big)}\rightarrow 0. \] 
Because of our choice of $k$, and the identification
$$\cW^{\geq -k}(-\ell E_i) \simeq \cW^{\geq 0} ( kE -\ell E_i), $$
the last term in the exact sequence has an induced action of $\Gamma_i + (k - \ell)\cdot {\rm Id}$,  with non-negative eigenvalues; see also (\ref{res}). (Note that since the first term in the sequence comes from an intersection of lattices, it is preserved by the residue action on $\cW|_{E_i}$.) Since this action is induced by the action of $\Gamma_i$ on $E_i$,  by definition it follows that 
$$ - k + \ell \ge t.$$
We need to show that $-k \ge t$. If on the contrary we assume that $k + t > 0$, it follows that $\ell > 0$ as well. But the exact same argument 
can be run for every $E_i$, and so it follows that 
$$\cW \subseteq \bigcap_i \cW^{\ge - k} (- E_i) = \cW^{\ge - k + 1},$$
which contradicts the minimality of $k$.
\end{proof}

\begin{remark}
Although not needed for the statement above, it is worth noting that the argument above can be continued inductively 
in order to reconstruct the entire minimal polynomial of $\Gamma_i$ acting on $\cW|_{E_i}$. Roughly speaking, denoting 
\[\cW_1:=\cW\cap\cW^{\geq -k} \big(-(\ell +1)E_i\big),\]
one can repeat an appropriate procedure for $\cW_1$ instead of $\cW$. The process will eventually stop  for dimension reasons.
\end{remark}

To deal with the pull-back of the Hodge filtration, we use the following simple lemma (see also \cite[Lemma 5.1]{FF}):

\begin{lemma}\label{incl}
Let $\shE$ be a locally free sheaf on $X$, and let $\shF$ and $\shG$ be two subsheaves of $\shE$ such that $\shG$ is locally free, and of maximal rank at each point as a subsheaf of $\shE$ (i.e. $\shG$ is a subbundle). If $\shF|_U=\shG|_U$ for some nonempty Zariski open subset $U\subseteq X$, then $\shF\subseteq \shG$.
\end{lemma}

Putting together (\ref{placement}) and Lemma \ref{incl}, we conclude:

\begin{corollary}\label{hodf}
\[f^*F_p\cV^{\geq 0}\subseteq F_p\cV_1^{\geq 0}.\]
\end{corollary}

\noindent
Recall now that we set
\[K_p (\cV^{\geq 0}) :=\shH^{-n} \text{gr}^F_p\text{DR}_D(\cV^{\geq0}) 
\simeq 
{\rm ker} \big( {\rm gr}_p^F \cV^{\geq 0} \overset{\theta_p}{\longrightarrow} {\rm gr}_{p+1}^F \cV^{\geq 0} \otimes \Omega_X^1 ({\rm log}~D)\big). \]

\begin{theorem}\label{WP_Deligne}
$K_p (\cV^{\geq 0})^{\vee}$ is weakly positive for any $p$.
\end{theorem}
\begin{proof}
As in the proof of  \corollaryref{WP_log},  we consider the natural diagram
$$
\begin{tikzcd}
f^*{\rm gr}_p^F \cV^{\geq 0}  \rar{f^*\theta_p} \drar{\theta^\prime_p} &  f^*{\rm gr}_{p+1}^F \cV^{\geq 0} \otimes f^* \Omega_X^1({\rm log}~D)  \dar \\
& f^*{\rm gr}_{p+1}^F \cV^{\geq 0} \otimes \Omega_{Y}^1({\rm log}~E)
\end{tikzcd}
$$
By \corollaryref{hodf} we have an inclusion 
$$\text{Ker} ~\theta'_p \subseteq K_p (\cV_1^{\geq 0})$$
which is generically an isomorphism. Using \corollaryref{WP_log} for $\cV_1^{\geq 0}$, together with \lemmaref{wplemma}(1), it follows that $(\text{Ker} ~\theta'_p)^\vee$ is weakly positive, hence so is 
$(\text{Ker} ~f^*\theta_p)^\vee$, again by \lemmaref{wplemma}(1). Therefore $K_p (\cV^{\geq 0})^\vee$ is also weakly positive because of \lemmaref{wplemma}(3).
\end{proof}

Using the notation $K_p (M)$  for the kernels associated to the underlying filtered $\Dmod$-module, we can now deduce the main result in the setting of this section:

\begin{corollary} \label{wp}
$K_p (M)^{\vee}$ is weakly positive for any $p$.
\end{corollary}
\begin{proof}\footnote{We thank both the referee and C. Schnell for suggesting this as a replacement for an earlier  
argument that needed more justification.}
The filtered quasi-isomorphism \eqref{qiso1} induces an isomorphism 
\[K_p (\cV^{\geq 0})  \simeq  K_p (j_*j^{-1}M). \] 
On the other hand, by definition there is a natural morphism
\[K_p (M) \longrightarrow K_p (j_*j^{-1}M)\]
which is an isomorphism over $U=X \smallsetminus D$.
Passing to duals, by \theoremref{WP_Deligne} and \lemmaref{wplemma}(1) we obtain that the dual of $K_p (M)$ 
is also weakly positive.
\end{proof}

\begin{remark}
Though not necessary for the argument, it is worth noting, as C. Schnell has pointed out to us, that using the $V$-filtration axioms for Hodge modules one can check (independently of the normal crossings hypothesis) 
that $K_p (M) \simeq K_p (j_*j^{-1}M )$ for all $p$ as well. Thus the morphism appearing the proof of \corollaryref{wp}
is in fact an isomorphism. 
\end{remark}

\subsection{General case}\label{general}
We now assume 
\[\bV=(\cV, \fddot, \V_\Q)\] 
to be a polarizable variation of Hodge structure with quasi-unipotent local monodromies, defined on open subset $U\subset X$ such that $D=X -  U$ is an arbitrary divisor. 
As in Section \ref{normal_crossings}, we denote the pure  Hodge module extension of $\bV$ by $M$. 

\begin{proof}[Proof of \theoremref{thm:WP}]
Suppose 
\[f: (X', E)\rarr (X, D) \] is a log resolution of the pair $(X, D)$ which is an isomorphism over $U$, with 
$E=(f^*D)_{\text{red}}$. We also consider the pure Hodge module $M'$ with strict support $X'$ extending $\bV$, this time seen as a variation of Hodge structure on $X' -  E$.

First, by the Stability and Decomposition Theorem for pure Hodge modules \cite[Th\'eor\`eme 5.3.1]{Saito-MHP}, 
we have that each $\cH^i f_*M'$ is a pure Hodge module on $X$, while  the underlying filtered 
$\Dmod_X$-modules satisfy
\[f_+ (\Mmod^\prime, F_{\bullet}) = \bigoplus_i \cH^i f_{+} (\Mmod', F_{\bullet}) [-i]\] 
in the derived category. (Here $f_+$ is the derived direct image functor for filtered left $\Dmod$-modules.)
Moreover, $M$ is a direct summand of $\cH^0 f_*M'$; it is in fact its component with strict support $X$.
By the commutation of direct images with the de Rham functor  \cite[\S2.3.7]{Saito-MHP} (see also \cite[Theorem 28.1]{Schnell2}), 
we have
\[\derR f_*\text{gr}^F_p\text{DR}(M')\simeq\text{gr}^F_p\text{DR}(f_*M')\simeq\bigoplus_i\text{gr}^F_p\text{DR}(\cH^i f_*M')[-i].\]
To simplify the notation, write
\[\shF:=\cH^{-n} \text{gr}^F_p\text{DR}(M'),\] and 
\[\shG:=\cH^{-n} \text{gr}^F_p\text{DR}(M).\] 
As we are looking at the lowest non-zero cohomology of the complexes in question, an easy argument involving the spectral sequence computing $\derR f_*\text{gr}^F_p\text{DR}(M')$  then shows:
\[f_*\shF \simeq \cH^{-n} \bigoplus_i\text{gr}^F_p\text{DR}(\cH^i f_*M')[-i]=\bigoplus_{i+j=-n}\cH^i \text{gr}^F_p\text{DR}(\cH^j f_*M').\] 
and therefore $\shG$ is a direct summand of $f_*\shF$. It suffices then to show that $(f_*\shF)^\vee$ is weakly positive. 
Since $\shF^\vee$ is weakly positive on $X^\prime$ by Corollary \ref{wp}, this is a consequence of the general Lemma \ref{WP_dual} below.\footnote{We thank the referee for suggesting this lemma and its proof; our original argument was unnecessarily complicated.} 
\end{proof}

\begin{lemma}\label{WP_dual}
Let $f \colon X^\prime \rightarrow X$ be a birational morphism of smooth projective varieties, and let $\shF$ be a coherent 
sheaf on $X^\prime$ such that $\shF^\vee$ is weakly positive. Then $(f_* \shF)^\vee$ is weakly positive as well. 
\end{lemma}
\begin{proof}
For any other sheaf $\shG$ on $X^\prime$, it is an immediate consequence of the definitions that there is a natural 
homomorphism of $\shO_X$-modules
$$f_* \mathcal{H}om_{X^\prime} (\shF, \shG) \longrightarrow \mathcal{H}om_X (f_*\shF, f_*\shG).$$
If in particular we take $\shG = \shO_{X^\prime}$, due to the fact that $f$ is birational morphism of smooth varieties
we have  $f_* \shO_{X^\prime} \simeq \shO_X$, and consequently this gives a homomorphism 
$$\varphi \colon f_* (\shF^\vee) \longrightarrow (f_* \shF)^\vee.$$
Since $f$ is generically an isomorphism, so is $\varphi$, and so Lemma \ref{wplemma}(1)  says that $(f_* \shF)^\vee$
is weakly positive provided $f_* (\shF^\vee)$ is so. But this holds by Lemma \ref{wplemma}(2).
\end{proof}

\noindent
{\bf Acknowledgement.} 
We are grateful to Christian Schnell for corrections and comments on a first draft. We also thank Junecue Suh for useful discussions, and to Yohan Brunebarbe for sharing a preprint of his work. Finally, our thanks go to the referee for extremely helpful 
corrections and suggestions for improvement.

\end{document}